\documentclass{amsproc}

\usepackage[all,2cell]{xy}
\usepackage{multirow}

\newtheorem{thm}{Theorem}[section]
\newtheorem{lem}[thm]{Lemma}
\newtheorem{cor}[thm]{Corollary}
\newtheorem{prop}[thm]{Proposition}

\theoremstyle{definition}

\newtheorem{defn}[thm]{Definitions}

\newtheorem{notn}[thm]{Notation}

\newtheorem{rem}[thm]{Remark}

\newtheorem{ass}[thm]{Assumptions}
\newtheorem{ex}[thm]{Example}

\def\C{{\mathbb C}}

\def\G{{\mathbb G}}

\def\P{{\mathbb P}}
\def\Q{{\mathbb Q}}
\def\R{{\mathbb R}}
\def\Z{{\mathbb Z}}

\def\cC{{\mathcal C}}

\def\cE{{\mathcal E}}
\def\cF{{\mathcal F}}

\def\cM{{\mathcal M}}
\def\cN{{\mathcal{N}}}
\def\cO{{\mathcal{O}}}

\def\cQ{{\mathcal{Q}}}

\def\cU{{\mathcal U}}

\def\Q{{\mathbb{Q}}}
\def\G{{\mathbb{G}}}

\def\f{\varphi}

\def\f{\varphi}

\def\lra{\longrightarrow}

\def\ra{\rightarrow}
\def\lra{\longrightarrow}

\def\operatorname#1{\mathop{\rm #1}\nolimits}

\def\Pic{\operatorname{Pic}}

\def\im{\operatorname{Im}}

\def\deg{\operatorname{deg}}

\def\det{\operatorname{det}}

\def\lin{\operatorname{lin}}

\def\qed{\hspace{\fill}$\rule{2mm}{2mm}$}
\def\NE{{\operatorname{NE}}}
\def\Nef{{\operatorname{Nef}}}

\def\Eff{{\operatorname{Eff}}}
\newcommand{\cNE}[1]{\overline{\NE(#1)}}
\newcommand{\cEff}[1]{\overline{\Eff(#1)}}

\def\arg{{\operatorname{arg}}}

\def\ch{{\operatorname{ch}}}
\def\td{{\operatorname{td}}}
\def\Sl{{\operatorname{Sl}}}
\def\Sp{{\operatorname{Sp}}}
\def\Gg{{\rm{G}}}

\newcommand{\shse}[3]{0 ~\ra ~#1~ \lra ~#2~ \lra ~#3~ \ra~ 0}

\begin{document}

\title{A classification theorem on Fano bundles}

\author{Roberto Mu\~noz}
\address{Departamento de Matem\'atica Aplicada, ESCET, Universidad
Rey Juan Carlos, 28933-M\'ostoles, Madrid, Spain}
\email{roberto.munoz@urjc.es}
\thanks{First and third author partially supported by the spanish government project MTM2009-06964}

\author{Gianluca Occhetta}
\address{Dipartimento di Matematica, Universit\'a di Trento, via
Sommarive 14 I-38123 Povo (TN), Italy}
\email{gianluca.occhetta@unitn.it}

\author{Luis E. Sol\'a Conde}
\address{Departamento de Matem\'atica Aplicada, ESCET, Universidad
Rey Juan Carlos, 28933-M\'ostoles, Madrid, Spain}
\email{luis.sola@urjc.es}



\begin{abstract}
In this paper we classify rank two Fano bundles $\cE$ on Fano manifolds satisfying $H^2(X,\Z)\cong H^4(X,\Z)\cong\Z$.
The classification is obtained via the computation of the nef and pseudoeffective cones of the projectivization $\P(\cE)$, that allows us to obtain the cohomological invariants of $X$ and $\cE$. As a by-product we discuss Fano bundles associated to congruences of lines, showing that their varieties of minimal rational tangents may have several linear components.
\end{abstract}
\maketitle


\section{Introduction}\label{sec:intro}

One of the most important open problems in the theory of vector bundles on complex projective manifolds is the existence of indecomposable vector bundles of low rank on the projective space.
In fact, the extense literature on the subject shows that, even for rank two, the problem is complicated.
One way in which one may try to tackle it is by studying the space of numerical classes of divisors on the projectivization $\P(\cE)$ of the bundle, and the cones (nef, pseudoeffective) contained in it.
In our previous paper \cite{MOS3} we have considered this type of argumentation in the setting of rank two vector bundles on Fano manifolds $X$ with $H^2(X,\Z)\cong H^4(X,\Z)\cong \Z$, obtaining a number of results regarding the structures of the nef and pseudoeffective cones of $\P(\cE)$ and their relation with the decomposability of the bundle, as well as some applications.

On the other hand, if $\cE$ is a {\it Fano bundle}, i.e. if $\P(\cE)$ is a Fano manifold, we have a second contraction $\pi':\P(\cE)\to X'$, that
we may use to infer important properties about $\cE$ (stability, for instance) and the cones of divisors of $\P(\cE)$; see Section \ref{sec:setup} below. In this paper we use these techniques
to classify rank two Fano bundles on manifolds with  $H^2(X,\Z)\cong H^4(X,\Z)\cong \Z$ (see Notation \ref{ssec:cast}):

\begin{thm}\label{thm:fanobdl}
Let $X$ be a Fano manifold satisfying $H^2(X,\Z)\cong H^4(X,\Z)\cong\Z$, and let $\cE$ be an indecomposable  rank two Fano bundle on $X$.
Then, up to a twist with a line bundle, $\cE$ is the pull-back of the universal quotient bundle on a Grassmannian $\G(1,m)$ by a finite map $\psi:X\to\G(1,m)$ where either

\begin{itemize}
\item $\psi$ is one of the embeddings given by
\begin{itemize}
\item[(P1)] $\P^2\cong \G(1,2)$,
\item[(P2)] $\Q^3\cong L\G(1,3)\subset\G(1,3)$,
\item[(P3)] $\P^3\cong \G(1,4)_{\Q^3}\subset \G(1,4)$,
\item[(P4)] $\Q^5\cong \G(1,13)_{K(\Gg_2)}\subset\G(1,13)$,
\item[(P5)] $K(\Gg_2)\cong F\G(1,6)\subset \G(1,6)$,
\item[(D1)] $\P^2\subset \G(1,4)_{\Q^3}$ (set of lines in $\Q^3$ meeting a fixed line),
\item[(D2)] $v_2(\P^2)\subset \G(1,3)$ (set of secant lines to $v_3(\P^1)\subset\P^3$),
\item[(D3)] $V_{5}^3\subset \G(1,4)$ (set of trisecant lines to the isomorphic projection of $v_2(\P^2)$ into $\P^4$),
\item[(C6)] $V_4^3\subset \G(1,3)$ (smooth quadric section of the Pl\"ucker embedding); or
\end{itemize}
\item $\psi$ factorizes by a finite covering $\psi_1:X\to X_1$ of one of the submanifolds of types {\rm (P1),(P2),(P3),(P4),(P5)} above and, either
\begin{itemize}
\item[(C1)] $X=X_1=\P^2$  and $\psi_1$ is given by a base point free two dimensional linear subsystem of conics, or
\item[(C2-5)] $X_1=\Q^3,\,\P^3,\,\Q^5$ or $K(\Gg_2)$, $X$ is a quadric section of the cone with vertex one point over the natural embedding of $X_1$ and $\psi_1$ is the projection form the vertex of the cone.
\end{itemize}
\end{itemize}
\end{thm}

Our assumptions are satisfied by many significative examples of Fano manifolds, including Fano threefolds with $\Pic(X) \simeq \Z$ (by Poincar\'e duality), projective spaces and quadrics (of dimension different from $2$ and $4$).
Furthermore we may construct many other examples upon these,
by considering appropriate subvarieties (by Barth--Lefschetz Theorems) or cyclic coverings (by \cite{Co}).
Rank two Fano bundles over projective spaces and quadrics were already completely classified (see \cite{SW1,SW2,SSW,APW}). With our approach we recover those classifications (clearly with the exception of $\Q^2$ and $\Q^4$, which do not satisfy our assumptions).

Note that we proved in \cite{MOS1} and \cite{MOS2} that, up to twists, the only indecomposable rank two Fano bundle in $\G(1,m)$ is the universal quotient bundle $\cQ$. A case by case analysis of the Fano bundles appearing in the above theorem  allows us to state the following:
\begin{cor}\label{cor:nef}
Let $X$ be a Fano manifold satisfying $H^2(X,\Z)\cong H^4(X,\Z)\cong\Z$, and let $\cE$ be an indecomposable  rank two Fano bundle on $X$. Then there exists an integer $m$ and a finite map $\psi:X\to\G(1,m)$ such that $\psi^*(\Nef(\P(\cQ)))=\Nef(\P(\cE))$.
\end{cor}

To our best knowledge it is still an open problem to find examples of rank two Fano bundles on manifolds of Picard number one in which this does not hold.

The paper is organized as follows. Section \ref{sec:setup} is devoted to the preliminary results about vector bundles we have obtained in \cite{MOS3}. In particular we recall the classification of Fano bundles $\cE$ whose projectivization $\P(\cE)$ has two different $\P^1$-bundle structures. If this is not the case, our hypothesis on $H^4(X,\Z)$ implies that $\P(\cE)$ has either a conic bundle structure or a divisorial contraction, according to a structure theorem (see \ref{lem:contractiontype}). In Section \ref{sec:ex} we present examples of Fano bundles
constructed upon subvarieties of Grassmannians, showing in particular that the list of Theorem \ref{thm:fanobdl} is effective.

We then study the two remaining scenarios separately: the existence of a divisorial contraction (Section \ref{sec:proofdiv}) or of a conic bundle structure (Section \ref{sec:proofconic}) on $\P(\cE)$. In each case the classification involves computing the possible discrete invariants of $\cE$ and the base manifold $X$, as well as some other ad hoc arguments.

Finally, we have included an appendix devoted to the relations between Fano bundles and congruences of lines, paying special attention to their varieties of minimal rational tangents. In particular we show that the VMRT of the congruence of trisecant lines to an isomorphic projection of a Severi variety has three linear components, providing negative answers to two problems posed by Hwang and Mok, see \ref{rem:hmconj}.

\subsection{Notation}\label{ssec:cast}
Throughout this paper we will work over the field of complex numbers. Let us introduce some notation regarding complex projective manifolds that will appear in this paper.

As usual, $\P^m$ will denote the projective space of dimension $m$, and $v_k$ the $k$-th Veronese embedding. A non singular quadric of dimension $m$ in $\P^{m+1}$ will be denoted by $\Q^m$, and a Del Pezzo manifold of degree $d$ and dimension $m$ by $V^m_d$.

The Grassmannian of lines in $\P^m$ will be denoted by $\G(1,m)$. The universal quotient bundle on $\G(1,m)$ is the rank two vector bundle $\cQ$ whose determinant is the ample generator of $\Pic(\G(1,m))$ and whose projectivization is the universal family of lines in $\P^m$. Given a subvariety $M\subset\P^m$, $\G(1,m)_{M}$ stands for the subscheme of $\G(1,m)$ parametrizing lines contained in $M$. Note that for $m=2$, we have an isomorphism $\G(1,2)\cong\P^2$, so that $\P(\cQ)$ has two $\P^1$-bundle structures over $\P^2$; note also that in this case $\P(\cQ)$ is the complete flag of the Lie group $\Sl(3)$.

Finally we need to refer to two special subvarieties of grassmannians. Given a symplectic form $L$ in $\C^4$, the subvariety of $\G(1,3)$ parametrizing lines in $\P^3=\P(\C^4)$ that are isotropic with respect to $L$ is denoted by $L\G(1,3)$. It is well known that this variety is a linear section of (the Pl\"ucker embedding of) $\G(1,3)$, hence isomorphic to $\Q^3$. The evaluation morphism on the projectivization $\P(\cQ_{|\Q^3})$ defines a $\P^1$-bundle structure over $\P^3$, which provides an isomorphism $\P^3\cong\G(1,4)_{\Q^3}$. In this case $\P(\cQ_{|\Q^3})$ is the complete flag of the Lie group $\Sp(4)$.

We have a similar situation in the case of the group $\Gg_2$. Up to the choice of a maximal abelian subgroup, this group has two maximal parabolic subgroups, whose corresponding quotients may be identified with the contact homogeneous manifold $K(\Gg_2)\subset\P^{13}$ (that appears as the closed orbit of the adjoint representation of the group $\Gg_2$), and the quadric $\Q^5\subset\P^6$. The quotient $F(\Gg_2)$ of $\Gg_2$ by the Borel subgroup has two structures of $\P^1$-bundle that allow us to describe it as the universal family of lines parametrized by $\G(1,13)_{K(\Gg_2)}\cong\Q^5$, or as the universal family of lines in $\P^6$ that are isotropic with respect to a non-degenerate $4$-form in $\C^7$ (which is parametrized by a submanifold $F\G(1,6)\cong \G(1,6)_{\Q^5}$, isomorphic to $K(\Gg_2)$). The only rank two vector bundle over $\Q^5$ with $c_1=-1$ whose projectivization is $F(\Gg_2)$ is classically known as the {\it Cayley bundle} (cf. \cite{O}).


\section{Setup and preliminary results}\label{sec:setup}

In this paper we will study pairs $(X,\cE)$ satisfying the following assumptions:

\begin{ass}\label{ass:setup}
$X$ is a complex Fano manifold of dimension $n$ such that $H^2(X,\Z)\cong H^4(X,\Z)\cong\Z$, and $\cE$ is a rank $2$ indecomposable Fano bundle on $X$.
\end{ass}

We will write $H_X$ for the ample generator of $\Pic(X)$, $\Sigma$ for the positive generator of $H^4(X,\Z)$, and we will set $-K_{X}=iH_X$, $H_X^2=d\Sigma$, $i,d\in\Z$. In the same way, the Chern classes of $\cE$ will be identified with integers $c_1$ and $c_2$. The Grothendieck projectivization of $\cE$ will be denoted by $Y:=\P(\cE)$. The natural projection from $Y$ to $X$ will be denoted by $\pi$, its general fiber by $f$, the divisor associated to the tautological divisor on $Y$ by $L$, the pull-back $\pi^*H_X$ by $H$ and the relative canonical divisor $K_{Y}-\pi^*K_{X}=-2L+c_1H$ by $K$. 

A {\it Fano manifold} is a smooth projective variety whose anticanonical line bundle is ample. We have assumed that $\cE$ is a {\it Fano bundle}, i.e. that $Y$ is a Fano manifold as well. Equivalently,
$$
0\leq\tau=\tau(\cE):=\inf\left\{t\in\R\,:\,-K+tH \mbox{ is ample}\right\}\leq i.
$$
Since checking this condition on bundles of the form $\cO_X(aH_X)\oplus\cO_X(bH_X)$ is immediate, we have restricted ourselves to {\it indecomposable} bundles (i.e. that are not a direct sum of line bundles).

Note that, by definition, $-K+\tau H$ generates an extremal ray of the {\it nef cone }$\Nef(Y)$ inside of the vector space of numerical classes of $\R$-divisors $N^1(Y)_\R$. Being $Y$ a Fano manifold, generalities on Mori theory tell us that $-K+\tau H$ is a semiample $\Q$-divisor, associated to a contraction $\pi':Y\to X'$. 
The following result follows from \cite[Theorems~2.3,~6.1,~6.3]{MOS3}:

\begin{lem}\label{lem:MOS1}
Let $(X,\cE)$ be a pair satisfying Assumptions \ref{ass:setup}. With the same notation as above:
\begin{enumerate}
\item $\tau$ is strictly positive,
\item every fiber of $\pi'$ has dimension smaller than or equal to one,
\item $\cE$ is stable unless $X=\P^2$ and $Y$ is the blow up of $\Q^3$ along a line, and
\item the discriminant $\Delta:=c_1^2-4c_{2}/d$ is strictly negative.
\end{enumerate}
\end{lem}

Denote by $f'$ a rational curve  in $Y$ of minimum  anticanonical degree between curves contained in fibers of $\pi'$. As a consequence of the second item of Lemma \ref{lem:MOS1}, we may apply Theorem 1.2 of \cite{Wis} (see also \cite[Lemma 6.2]{MOS3}) in order to obtain the following result:

\begin{lem}\label{lem:contractiontype} With the same notation as above, then:
\begin{itemize}
 \item[{{\rm (P)}}] either $\pi'$ is a $\P^1$-bundle, $ -K_Y\cdot f'= 2$, or
 \item[{{\rm (D)}}] $\pi'$ is the blow-up of a codimension two smooth subvariety, $-K_Y\cdot f'= 1$, or
\item[{{\rm (C)}}] $\pi'$ is a conic bundle with reducible fibers, $ -K_Y\cdot f'= 1$.
\end{itemize}
In all cases $X'$ is smooth and Fano of Picard number one, and
$$\tau=\frac{K_Y\cdot f'}{H \cdot f'}+i=\frac{K\cdot f'}{H \cdot f'}>0.$$
\end{lem}
This result suggests considering separately the three cases, to which we will refer to as {\it pairs $(X,\cE)$ of types {\rm (P)}, {\rm (D)} or {\rm (C)}}. 

We will also consider the {\it pseudoeffective cone} of $Y$, denoted with $\cEff{Y}$, which is the closure in $N^1(Y)_\R$ of the cone generated by classes of effective divisors. It is generated by the classes of $H$ and of a certain $\R$-divisor of the form $-K+\rho H$, for some $\rho\in\R$ (cf. \cite[Proposition~4.12]{MOS3}): 
 
\begin{lem}\label{lem:rhotau} 
With the same notation as in Lemma \ref{lem:contractiontype}, it follows that:
\begin{enumerate}
\item if $(X,\cE)$ is of type {\rm (P)} or {\rm (C)}, then $\rho=\tau$;
\item if $(X,\cE)$ is of type {\rm (D)} then $0\leq\rho<\tau$ and $-K+\rho H$ is numerically proportional to the exceptional divisor of $\pi'$. 
\end{enumerate} 
 Moreover, in every case $\rho$ and $\tau$ are related by the equation:
\begin{equation}
\label{eq:rhotau}
n\,\arg\left(\tau+\sqrt{\Delta}\right)+\left(\rho+\sqrt{\Delta}\right)=\pi
\end{equation}
In particular, in cases {\rm (P)} and {\rm (C)} $n$ is necessarily equal to $2,3$ or $5$.
\end{lem}

Denote by $H_{X'}$ the ample generator of $\Pic(X')$, by $H'$ its pull-back to $X'$, by $i'$ the index of $X'$ and by $K':=K_Y-\pi'^*K_{X'}$ the relative canonical divisor, and set
$$
\tau':=\inf\left\{t\in\R\,:\,-K'+t\pi'^*H'\mbox{ is ample}\right\}.
$$

Consider moreover the following positive integers
$$\lambda:=-K_Y\cdot f',\quad\mu:=H\cdot f',\quad\mu':=H'\cdot f,\quad \nu:=K\cdot f',\quad \nu':=K'\cdot f, $$ and note that:
\begin{equation}\label{eq:numu1}
i\mu-\nu=-K_Y\cdot f'=\lambda\quad\mbox{ and }\quad i'\mu'-\nu'=-K_Y\cdot f=2,
\end{equation}
so that we have the following intersection numbers:

\begin{table}[h!]
\begin{tabular}{|c|c|c|c|c|}
\hline
&$-K$&$H$&$-K'$&$H'$\\
\hline
$f$&$2$&$0$&$-\nu'$&$\mu'$\\
\hline
$f'$&$-\nu$&$\mu$&$\lambda$&$0$\\
\hline
\end{tabular}
 \end{table}

We may then write:

\begin{equation}\label{eq:codim1}
 \begin{pmatrix}-K\\H\end{pmatrix}=
 \begin{pmatrix}2&-\nu\\0&\mu\end{pmatrix}
 \begin{pmatrix}-\nu'&1\\\mu'&0\end{pmatrix}^{-1}
 \begin{pmatrix}-K'\\H'\end{pmatrix}
 =A\begin{pmatrix}-K'\\H'\end{pmatrix},
\end{equation}
where:
$$
A:=\begin{pmatrix}\vspace{0.2cm}
-\dfrac{\nu}{\lambda}&\dfrac{2\lambda-\nu\nu'}{\lambda\mu'}\\
\dfrac{\mu}{\lambda}&\dfrac{\mu\nu'}{\mu'\lambda}
\end{pmatrix}.
$$

Note that the sets $\left\{-K,H\right\}$ and  $\left\{-K',H'\right\}$ generate subgroups of $\Pic(Y)$ of indices two and $\lambda$, respectively. In particular $(-2/\lambda)(\mu/\mu')=\det(A)=\pm 2/\lambda$ and we get
\begin{lem}\label{lem:mumu'}
With the same notation as above, the minimal $H'$-degree of a curve contracted by $\pi$ and the minimal $H$-degree of a curve contracted by $\pi'$, are equal.
\end{lem}

The following proposition gives the complete list of Fano bundles whose projectivization has a second $\P^1$-bundle structure. Note that the list corresponds to items (P1) to (P5) in Theorem \ref{thm:fanobdl}:

\begin{prop}\label{prop:typeP}
With the same notation as above, assume that $(X,\cE)$ is of type {\rm (P)}. Then $Y$ is isomorphic to $G/B$ where $G$ is a simple Lie algebra of type $A_2$, $B_2$ or $\Gg_2$, and $B$ is its Borel subgroup. Each contraction of $Y$ corresponds to one of the two possible choices of a maximal parabolic subgroup of $G$. 
\end{prop}

Let us sketch the proof of this result here; we refer the interested reader to \cite[Theorem.~6.5]{MOS3} for further details.

\begin{proof}
By assumption we may consider $Y$ as the projectivization of a Fano bundle $\cE'$ over $X'$; let us denote its discriminant by $\Delta'$. Using equation (\ref{eq:codim1}), the integer $-KH^n$ can be written in terms of $-K'$ and $H'$. Since $H'^n=0$ and $K'^2=\Delta'H'^2$, we may reduce this polynomial to a multiple of $-K'H'^n$:
$$
-KH^n=\left(\dfrac{\mu}{2}\right)^{n-1}\dfrac{\im\left(\tau'+\sqrt{\Delta'}\right)}{\sqrt{-\Delta'}}(-K'H'^{n})
$$ 
Note that $\Delta<0$ by Lemma \ref{lem:MOS1}. Moreover, applying Equation (\ref{eq:rhotau}) from Lemma \ref{lem:rhotau} one may rephrase the above equality as:
$$
-KH^n=\left(\dfrac{\mu\tau'}{2\cos\left(\frac{\pi}{n+1}\right)}\right)^{n-1}(-K'H'^{n}).
$$
Joining this formula with its symmetric, obtained by writing $-K'H'^n$ in terms of $-KH^n$, we get:
$$
\mu^2\tau\tau'=4\cos^2\left(\frac{\pi}{n+1}\right)=1,2,3\mbox{, for }n=2,3,5\mbox{, respectively.}
$$
From this one may easily obtain that, up to exchange  $X$ and $X'$, $X$ is either $\P^2\cong\G(1,2)$, $\P^3=\G(1,4)_{\Q^3}$ or $\Q^5=\G(1,13)_{K(\Gg_2)}$, and that the Chern classes of $\cE$ coincide with those of the corresponding universal quotient bundles. We conclude by noting that $\cE$ is stable by Lemma \ref{lem:MOS1} and that these universal bundles are determined by their Chern classes among stable bundles (cf. \cite[8.1]{H}, \cite[Lemma~4.3.2]{OSS}, \cite{O}). 
\end{proof}

Throughout the rest of the paper we will always assume that $\lambda=1$, i.e. that $(X,\cE)$ is of type {\rm (D)} or {\rm (C)}. Hence we may write:
\begin{equation}\label{eq:H2DC}
A=\begin{pmatrix}\vspace{0.2cm}
-\nu&-\dfrac{\nu\nu'-2}{\mu}\\
\mu&{\nu'}
\end{pmatrix}, \quad A^{-1}=\begin{pmatrix}\vspace{0.2cm}
-\dfrac{\nu'}{2}&\dfrac{2-\nu\nu'}{2\mu}\\
\dfrac{\mu}{2}&\dfrac{\nu}{2}
\end{pmatrix}.
\end{equation}
 We finish this section by stating an straightforward lemma that will be useful later.

\begin{lem}\label{lem:muodd}
With the same notation as above, it follows that:
\begin{enumerate}
\item if $(X,\cE)$ is of type {\rm (D)} or {\rm (C)}, then $c_1-i$ and $\mu$ are odd, and
\item if $(X,\cE)$ is of type {\rm (D)}, then 
\begin{equation}\label{eq:rho}
\rho=\dfrac{\nu\nu'-2}{\mu\nu'}.
\end{equation}
\end{enumerate}
\end{lem}

\begin{proof}
The first assertion follows from Equation (\ref{eq:numu1}) and the fact that $(-K+c_1H)/2=L$ is an integral divisor, so that
$$
\left(\dfrac{1}{2}\quad\dfrac{c_1}{2}\right)A=\left(\dfrac{c_1\mu-\nu}{2}\quad\dfrac{c_1\nu'\mu-\nu\nu'+2}{2\mu}\right)\in\Z^2.
$$
The second follows from the definition of $\rho$, since in case {\rm (D)} we have $E=K'$.
\end{proof}


\section{Examples}\label{sec:ex}

This section is devoted to the construction of examples of rank two Fano bundles on Fano manifolds satisfying Assumption \ref{ass:setup}. Since it is well known that a rank two vector bundle on a projective manifold may be regarded as the pull-back of a twist of the universal quotient bundle on a grassmannian of lines, the natural way of constructing examples of that kind is by considering special subvarieties of Grassmannians $\G(1,m)$, as well as finite coverings of them. 

\begin{ex}\label{ex:compint} \textbf{Complete intersections in Grassmannians}.\\
Let $X$ be a complete intersection of $k$ general hypersurfaces of degrees $d_1,\dots,d_k$ in the Pl\"ucker embedding of $\G(1,m)$, and assume that $k\leq 2m-5$ and $d_1+\dots+d_k< m$. Then $X$ is, by Lefschetz Theorem, a Fano manifold of Picard number one, and by adjunction, the restriction $\cQ_{|X}$ is a Fano bundle on $X$. Moreover, by \cite[Thm. 3.1]{Ma}, $\cQ_{|X}$ is semistable of degree one, hence it is also indecomposable.

a) If $k\leq m-2$, then the restriction of the evaluation map $\pi':\P(\cQ_{|X}) \to \P^m$ is a fiber type contraction, and the nef cones of $\P(\cQ_{|X})$ and $\P(\cQ)$ coincide. The general fiber of $\pi'$ is a complete intersection of type $(d_1, \dots, d_k)$ in $\P^{m-1}$.
On the other hand, by Lefschetz Theorem again, this procedure cannot provide examples of Fano bundles over manifolds with $H^4(X,\Z)\cong\Z$ unless $k=2m-5$. More concretely, one obtains precisely the submanifolds listed as (P2) and (C6) in Theorem \ref{thm:fanobdl}.

b) If $k=m-1$ then $d_1= \dots = d_k=1$ and $X$ is a general linear section of $\G(1,m)$. But then the map $\pi'$ is determined by a morphism $\cO_{\P^m}^k\to\Omega_{\P^m}(2)$, whose degeneracy locus has been described by Bazan and Mezzetti (see \cite[Section 2]{BM} and \cite[Proposition 2.4]{DP}): using their results it follows that in this case  $\pi':\P(\cQ_{|X}) \to \P^m$
is a divisorial contraction, and the image of the exceptional divisor has codimension two (it is either a rational variety if $m$ is even or the closure of a scroll parametrized by an open set of a hypersurface in $\P^{m-2}$ if $m$ is odd). 
Finally, imposing the condition $H^4(X,\Z)\cong \Z$ leaves us with the case $m=4$ and $k=3$, which is case
(D3) in Theorem \ref{thm:fanobdl}.\qed
\end{ex}

In general, given any finite morphism $\psi:X\to\G(1,m)$ from a Fano manifold $X$ of Picard number one, the bundle $\psi^*(\cQ)$ is Fano if $c_1(\psi^*(\cQ))$ is smaller than the index $i$ of $X$. In the next example we will construct subvarieties of Grasmannians satisfying this property that are not complete intersections, with special emphasis on those satisfying $H^4(X,\Z)=\Z$. 

\begin{ex}\label{ex:ncompit}\textbf{Other subvarieties of Grassmannians}.

a) The subvarieties described in (P3), (P4) and (P5) of Theorem \ref{thm:fanobdl} provide the first examples 
of Fano bundles constructed upon subvarieties of Grassmannians that are not complete intersections. Note that in the three cases the degree of the restriction of $\cQ$ equals $i-2$, hence we obtain more examples of Fano bundles by considering general hyperplane sections $X$ of the manifolds (P3), (P4) and (P5).
However, the only case satisfying $H^2(X,\Z)\cong H^4(X,\Z)\cong\Z$ (see Lemma \ref{lem:HKG2} below) is
 the hyperplane section of (P3):  this is (D1), classically known as the Veronese surface parametrizing the set of lines in $\Q^3$ meeting a fixed one $\ell$. 
 
Note that the manifolds (P3) and (P5) are examples of {\it linear congruences of lines}, i.e.  subvarieties of dimension $m-1$ of $\G(1,m)$, obtained by cutting $\G(1,m)$ with a (not necessarily general) linear space. To our best knowledge, the classification of linear congruences of lines with Picard number one is still open. See Appendix \ref{sec:appendix} for more details.
 
b) In the same manner, Quadric Grasmannians $\G(1,m)_{\Q^{m-1}}$ and Isotropic Grassmanians $L\G(1,m)$, as well as some of their complete intersections, provide examples of Fano submanifolds of Grassmannians of Picard number one, on which the corresponding restrictions of $\cQ$ are Fano bundles.  Moreover, apart of the obvious exceptions, they are not complete intersections in Grassmannians. However, they do not provide new examples satisfying Assumption \ref{ass:setup}.

c) Case (D2) in Theorem \ref{thm:fanobdl} corresponds to the embedding of a Veronese surface $v_2(\P^2)$ in $\G(1,3)$, which is clearly not a complete intersection. It may be described as the family of secant lines to a rational normal cubic $\Gamma_3\subset\P^3$. In fact, since $\Gamma_3$ has no trisecants, it follows that the family is parametrized by the second symmetric power of $\Gamma_3$, that is $\P^2$, and it is easy to see that the restriction of the Pl\"ucker embedding to this $\P^2$ is the complete linear system $|\cO_{\P^2}(2)|$. Set, as usual, $\cE:=\cQ_{|v_2(\P^2)}$.  Since for every point $P$ in $\P^3\setminus\Gamma_3$ there exists a unique line secant to $\Gamma_3$ passing by $P$, then the evaluation morphism $\pi':\P(\cE)\to\P^3$ is the blow-up of $\P^3$ along $\Gamma_3$. Its exceptional divisor $E$ is isomorphic to $\Gamma_3\times\Gamma_3$ and the restriction $\pi_{|E}$ is the natural quotient onto $\P^2$, which is a two-to-one finite morphism. \qed
\end{ex}

Finally, we will show some examples of Fano bundles that arise via non injective morphisms $\psi^*:X\to\G(1,m)$.

\begin{ex}\label{ex:P2}\textbf{Finite coverings}.

a) Cyclic coverings. Let $X_1$ be a Fano manifold with $H^2(X_1,\Z)\cong\Z$ and $\cE_1$ an indecomposable Fano bundle over $X_1$. Let $\psi:X \to X_1$ be the $d$-cyclic covering of $X_1$ determined by the ample generator $H_{X_1}$ of $\Pic(X_1)$. Let $Y_1$ and $Y$ be the projectivizations of $\cE_1$ and $\cE:=\psi^*\cE_1$, respectively, and let $H_1$ be the pull-back of $H_{X_1}$ to $Y_1$. Since $-K_Y = \psi^*(-K_{Y_1} -(d-1)H_1)$, then $\cE$ is Fano whenever $-K_{Y_1} -(d-1) H_1$ is ample or, equivalently, if $\tau(\cE_1)+d-1$ is smaller than the index of $X_1$. 

For instance, starting from (P1), (P2), (P3), (P4) or (P5) the only possibility is $d=2$ and the output of this process can be described in terms of the following geometric construction. Consider the second $\P^1$-bundle structure of $Y_1$, $\pi_1':Y_1=\P(\cE_1')\to X_1'$, and let $H_{X_1'}$ be the ample generator of $\Pic(X_1')$.
Note that in each case 
the bundle $\cE'_1\otimes\cO_{X'_1}(sH_{X'_1})$ is globally generated and not ample for some $s$.
Setting
$$
\cE':=\cE'_1\otimes\cO_{X_1'}(sH_1')\oplus\cO_{X_1'},$$
the image of the second contraction of $P:=\P(\cE')$ is, by construction, the cone $\cC(X_1)$ of $X_1$ over a point $O$. Let $Y$ be a general divisor in the linear system $|\cO_P(2)|$, $\pi$ be the restriction of the second contractions of $P$, with image $X:=\pi(Y)$, and $\pi':Y\to X':=X_1'$ the restriction of the natural projection.
$$
\xymatrix{
X\ar@{^{(}->}[d]&Y\ar@{^{(}->}[d]\ar[l]_{\pi\hspace{0cm}}\ar[r]^{\hspace{0cm}\pi'}&X'\\
\cC(X_1)&P=\P(\cE')\ar[l]\ar[r]&X_1'\ar@{=}[u]\\
X_1\ar@{^{(}->}[u]&Y_1=\P(\cE_1')\ar@{^{(}->}[u]\ar[l]_{\pi_1\hspace{0.4cm}}\ar[r]^{\hspace{0.4cm}\pi_1'}&X_1'\ar@{=}[u]}
$$
 It follows that $Y$ is a $\P^1$-bundle over $X$ via $\pi$ and a conic bundle over $X'$ via $\pi'$.
Furthermore, the linear projection from the point $O$ defines a two-to-one morphism $\psi:X\to X_1$, so that $\pi:Y\to X$ is the fiber product of $\pi_1:Y_1\to X_1$ over $\psi$:
$$
\xymatrix{
X\ar[d]_\psi&Y\ar[l]_{\pi}\ar[d]\\
X_1&Y'\ar[l]_{\pi_1}}
$$
Using this construction we obtain cases
(C2), (C3), (C4) and (C5) out of the corresponding (P)'s. Starting from
(P1) we obtain a Fano bundle over $\Q^2\cong\P^1\times\P^1$.

Note that (C6) gives another example of a Fano bundle on $V_4^3$ of type (C). Although its Chern classes are the same as those of the bundle corresponding to case (C2),
the bundles themselves are different, since the corresponding second contractions are clearly different (see Section \ref{sec:proofconic} below).

b) Non-cyclic coverings. 
Let $P$ be the cartesian product $\P^2\times\P^2$, $Y\subset P$ be a general divisor on $P$ of type $(1,2)$ and $\pi$ and $\pi'$ be the restrictions to $Y$ of the natural projections.
$$
\xymatrix{\P^2\ar@{=}[d]&P=\P^2\times\P^2\ar[l]\ar[r]&\P^2\ar@{=}[d]\\
\P^2&Y\ar@{^{(}->}[u]\ar[l]_{\pi}\ar[r]^{\pi'}&\P^2}
$$
By construction $\pi$ is a $\P^1$-bundle and $\pi'$ is a conic bundle. Following \cite{SW1}, the $\P^1$-bundle structure on $Y$ is given by the vector bundle $\cE$ on $\P^2$ that appears as the cokernel of a surjective morphism
$
\cO_{\P^2}(-2)\to\cO_{\P^2}^3.
$
In other words, the manifold $Y$ is the incidence variety of the correspondence between points of $\P^2$ and conics of a $2$-dimensional base point free linear subsystem of $|\cO_{\P^2}(2)|$. This is case (C1) in Theorem \ref{thm:fanobdl}.
\qed
\end{ex}


\section{Rank two Fano bundles with a divisorial contraction}\label{sec:proofdiv}

The aim of this section is classifying pairs $(X,\cE)$ of type (D). 
Before beginning, we need to introduce some notation regarding codimension two cycles on $X$, $X'$ and $Y$. 
\begin{notn}\label{notn:D}
Let $T\subset X'$ be the center of the blow-up $\pi'$, and $E:=\pi'^{-1}(T)\subset Y$ the exceptional divisor.
Following \cite[p.~605]{GH}, $H^4(Y,\Z)$ is, on one hand, generated by the $\Z$-basis 
$$\left\{LH, H^2/d\right\}$$ 
and, on the other, naturally isomorphic to 
\begin{equation}\label{eq:codim2blow}
H^4(X',\Z)\oplus H^4(E,\Z)/H^4(T,\Z)\cong H^4(X',\Z)\oplus (-E_{|E})H^2(T,\Z).
\end{equation}
Note that the last isomorphism follows from Chern-Wu relation on the $\P^1$-bundle $E\to T$. In particular we obtain that $H^4(X',\Z)$ is freely generated by a (positive) cycle $\Sigma'$ and $H^2(T,\Z)$ is freely generated by the class of an ample divisor $H_T$ of the form $\frac{1}{b}{H'}_{|T}$, for some positive integer $b$. 
Let us also introduce the integer $d'$, defined by $H'^2=d'\Sigma'$.
\end{notn}

Note that this description  provides a second $\Z$-basis of $H^4(Y,\Z)$, namely 
$$\left\{-E H'/b,H'^2/d'\right\}.$$
The relation between the two $\Z$-bases is 
$$
\left(\begin{matrix}\vspace{0.3cm}-EH'/b\\H'^2/d'\end{matrix}\right)=B
\left(\begin{matrix}\vspace{0.3cm}LH\\ H^2/d\end{matrix}\right).
$$
where, using (\ref{eq:codim1}) and Chern-Wu relations:
$$
B:=\begin{pmatrix}\vspace{0.3cm}\displaystyle\frac{\nu\nu'-1}{b} &\displaystyle \frac{d}{4b\mu}\left(\nu'\mu^2\Delta+2c_1(1-\nu\nu')\mu+\nu(\nu\nu'-2)\right)
\\\displaystyle \frac{\nu\mu}{d'} & \displaystyle\frac{d}{4d'}(\Delta\mu^2-2c_1\nu\mu +\nu^2)\end{pmatrix}.
$$

As a first step towards classification, we will show that the images into $X$ and $X'$ of the fibers of $\pi'$ and $\pi$, respectively, have degree one with respect to the ample generator of the corresponding Picard group.

\begin{prop}\label{prop:linesdiv}
Let $(X,\cE)$ be a pair of type {\rm (D)}. With the same notation as in Section \ref{sec:setup}, it follows that $\mu=H\cdot f'=H'\cdot f=1$, and in particular $\tau=\nu\in\Z$ and $\tau'=\nu'\in\Z$.
\end{prop}

Set $\gamma:=\gcd(d,\mu),\,\bar{\mu}:=\mu/\gamma,\,\bar{d}:=d/\gamma$. In the proof of \ref{prop:linesdiv} we will make use of the following auxiliary result:

\begin{lem}\label{lem:gcd=1}
With the same notation as above, $\gcd(\bar{d},\mu)=1$.
\end{lem} 

\begin{proof}
Note first that all the entries of $B$ are integers, in particular $B_{1,1}=(\nu\nu'-1)/b\in\Z$. Since $\nu\nu'-1\equiv 1$ modulo $\mu$, it follows that 
\begin{equation}\label{eq:bmu}
\gcd(b,\mu)=1.
\end{equation}
The determinant of $B$ is a unit, hence we get
\begin{equation}\label{eq:det}
d\left(\nu^2-\Delta\mu^2\right)=\pm 4bd'.
\end{equation}
Since the left hand side is positive (by Lemma \ref{lem:MOS1}(4)) and $\nu\equiv 1$ modulo $\mu$, we get that
$d\equiv 4bd'$ modulo $\mu$ and, in particular 
\begin{equation}\label{eq:d1d2mu}
\gamma=\gcd(d,\mu)=\gcd(4bd',\mu)=\gcd(d',\mu).
\end{equation}
The last equality follows by Lemma \ref{lem:muodd} and (\ref{eq:bmu}).

Set $\bar d':=d'/\gamma$. Dividing by $\gamma$ in equation (\ref{eq:det}), the same argument as above provides: 
\begin{equation}\label{eq:d1d2mu2}
\gcd(\bar d,\mu)=\gcd(4b\bar d',\mu)=\gcd(\bar d',\mu).
\end{equation}

We conclude by showing that $\gcd(\bar d',\mu)=1$. Since $B_{2,1}=\nu\mu/d'=\nu\bar\mu/\bar d'$ is an integer and $gcd(\bar\mu,\bar d')=1$, it follows that $\nu/\bar d'$ is an integer, too. But $\nu\equiv 1$ modulo $\mu$, then $\bar d'$ and $\mu$ are coprime.
\end{proof}

\begin{proof}[Proof of Proposition \ref{prop:linesdiv}]
By Lemma \ref{lem:muodd} we have $\rho=\tau-\frac{2}{\nu'\mu}$, so Equation (\ref{eq:rhotau}) in Lemma \ref{lem:rhotau} reads as: 
$$
\arg\left(\left(\tau+\sqrt{\Delta}-\frac{2}{\nu'\mu}\right)(\tau+\sqrt{\Delta})^n\right)=\pi.
$$
In other words, we get:
$$
\nu'\mu\im\left((\tau+\sqrt{\Delta})^{n+1}\right)=2\im\left((\tau+\sqrt{\Delta})^n\right)
$$
Multiplying by $\mu^nd^{\lfloor \frac{n}{2}\rfloor}/\sqrt{-\Delta}$, all the terms of the corresponding expansions of both sides of this equation are integers. Furthermore, since $\nu_i \equiv -i \mod \mu$, if we take classes modulo $\mu$ we get: 
$$(-2)(-1)^n(n+1)\bar{d}^{\lfloor n/2\rfloor} \equiv (-1)^{n-1}2n\bar{d}^{\lfloor n/2\rfloor} \mod \mu$$
This implies, being $\mu$ odd by Lemma \ref{lem:muodd} and $\gcd(\bar{d},\mu)=1$ by Lemma \ref{lem:gcd=1}, that $\mu=1$.
\end{proof}

In order to complete the classification of Fano bundles of type (D), we will study the restricted morphism $\pi_{|E}:E\to X$. Note that, by Proposition \ref{prop:linesdiv} and the change of base (\ref{eq:codim1}) we may write
\begin{equation}\label{eq:E}
E=K'=\tau'L+\dfrac{\tau'(\tau-c_1)-2}{2}H,\quad \rho=\tau-\dfrac{2}{\tau'}
\end{equation}
and assert that $\pi_{|E}$ is a generically finite morphism of degree $\tau'$. The next result studies the case in which $\pi_{|E}$ is not finite.

\begin{prop}\label{prop:typeDnotfin}
Let $(X,\cE)$ be a pair of type {\rm (D)}. With the same notation as above, assume that $E$ contains a fiber of $\pi$. Then $(X,\cE)$ is of type {\rm (D1)}.
\end{prop}

\begin{proof}
Since $E$ contains a fiber $f$ of $\pi$ we can compute the intersection of $H_T$ with $\pi'(f)$, getting $1=H'f=b\pi'^*H_Tf$. In particular $b=1$.

Moreover, by Propostition \ref{prop:linesdiv}, we know that $\frac{ \tau}{d'}=B_{2,1}$ is an integer. Considering the determinant of $B$ one may write
\begin{equation}\label{eq:tau1}
4=4\det(B)=\dfrac{(\tau^2-\Delta)d}{d'} = B_{2,1}d\left(\tau -\dfrac{\Delta}{\tau}\right)<B_{2,1}d\tau\left(1+\tan^2\left(\dfrac{\pi}{n+1}\right)\right),
\end{equation}
where the last inequality follows from Lemma \ref{lem:rhotau}. It follows that:
\begin{itemize}
\item either $n=2$, $\tau=2$, $B_{2,1}d=1$ and $\Delta=-4$, or
\item $n\in\{3,4\}$, $\tau=1$, $B_{2,1}d=3$, $\Delta=-1/3$, or
\item $n\in\{3,4\}$, $\tau=3$, $B_{2,1}=d=1$, $\Delta=-3$.
\end{itemize}
Note also that, by (\ref{eq:E}) and Prop. \ref{prop:linesdiv}, 
we have $\rho=\tau-2/\tau'$, so equation (\ref{eq:rhotau}) reads as: 
$$
n\,\arg\left(\tau+\sqrt{\Delta}\right)+\left(\tau-\dfrac{2}{\tau'}+\sqrt{\Delta}\right)=\pi.
$$
Plugging the data above into this formula, we may compute the value of $\tau'$ in each case, obtaining the following possibilities:
\begin{table*}[!!h]
\centering
\begin{tabular}{|c|c|c|c|c|c|c|c|c|}\hline
$n$&$i$&$\tau$&$c_1$&$c_{2}$&$d$&$d'$&$\tau'$&$i'$\\\hline\hline
$2$&$3$&$2$&$0$&$1$&$1$&$2$&$1$&$3$\\\hline
$3$&$2$&$1$&$-1$&$1$&$3$&$1$&$2$&$4$\\\hline
\multirow{2}{*}{$4$}&$2$&$1$&$-1$&$1$&$3$&$1$&$3$&$5$\\\cline{2-9}
&$4$&$3$&$-1$&$1$&$1$&$3$&$1$&$3$\\\hline                                      
\end{tabular}
\end{table*}

Since the fourth Betti number of a $4$-dimensional quadric is two, neither $X$ nor $X'$ can be isomorphic to $\Q^4$, and we may rule out the rows two and four. In the first case $T$ is an irreducible smooth submanifold of $\Q^3$ containing the image by $\pi'$ of a fiber of $\pi$. Since $\mu=1$ by Proposition \ref{prop:linesdiv} it follows that $T$ is a line, and then it is straightforward that $(X,\cE)$ is of type (D1). Finally in the third case $X'\cong \Q^5$ and $X$ is a Mukai manifold whose degree is multiple of $d^2=9$, hence a hyperplane section of $K(\Gg_2)$ by \cite{Mu}. This contradicts that $H^4(X,\Z)\cong \Z$ by the following lemma. 
\end{proof}

\begin{lem}\label{lem:HKG2}
Let $X$ be a hyperplane section of the Pl\"ucker embedding of $K(\Gg_2)$ into $\P^{13}$. Then $H^4(X,\Z)\cong\Z^2$.
\end{lem}

\begin{proof}
Let $F(\Gg_2)$ be the complete flag associated to the Lie group $\Gg_2$, and let $\pi_1:F(\Gg_2)\to K(\Gg_2)$ and $\pi_1':F(\Gg_2)\to \Q^5$ be its two $\P^1$-bundle structures (see Section \ref{ssec:cast}). The latter is the projectivization of the Cayley bundle $\cC$ on $\Q^5$, so that $\pi_1$ corresponds to evaluation of global sections of the twist $\cC(2)$; this follows from the base change (\ref{eq:codim1}) (see also \cite[Theorem.~6.5]{MOS3}).\\ 
Then, the inverse image $Y=\pi_1^{-1}(X)$ may be identified with the blow-up of $\Q^5$ along the smooth zero locus of a general section of $\cC(2)$ which, by \cite[Thm.~3.7]{O}, is isomorphic to $\P(T_{\P^2})$. Since $H^2(\P(T_{\P^2}),\Z)\cong\Z^2$, Equation (\ref{eq:codim2blow}) tells us that $H^4(Y,\Z)\cong\Z^3$, and so $H^4(X,\Z)\cong\Z^2$.
\end{proof}

Let us finally consider the case in which $\pi_{|E}$ is finite. 
The main tool we will use is the following:

\begin{lem}{\cite[Thm.~2.1]{L}} \label{lem:laz}
Let $\varphi:M\to N$ be a finite morphism between smooth projective varieties, and let $\cF$ denote the cokernel of the natural inclusion $\cO_N\to\varphi_*\cO_M$, which is a vector bundle on $N$ of rank $\deg(\varphi)-1$. If $\cF^\vee$ is ample and $\dim (N)\geq\deg(\varphi)+1$, then $\varphi^*:H^2(N,\Z)\to H^2(M,\Z)$ is an isomorphism.
\end{lem}

Since we know that $H^2(E,\Z)\not\cong\Z\cong H^2(X,\Z)$, this result applied to $\pi_{|E}$ implies that:

\begin{lem}\label{lem:lazapp}
Let $(X,\cE)$ be a pair of type {\rm (D)}. With the same notation as above, if $\pi_{|E}$ is finite, then either $X'\cong\P^{n+1}$, or $i=2$.
\end{lem}

\begin{proof}
Considering the exact sequence
$$
\shse{\cO_{\P(\cE)}(-E)}{\cO_{\P(\cE)}}{\cO_E}
$$
and applying $\pi_*$ we obtain:
$$
\shse{\cO_{X}}{(\pi_{|E})_*\cO_E}{R^1\pi_*\big(\cO_{\P(\cE)}(-E)\big)}
$$
Setting $\cF:= R^1\pi_*\big(\cO_{\P(\cE)}(-E)\big)$ and using (\ref{eq:E}),
one gets:
$$
\cF^\vee\cong S^{\tau'-2}\cE\left(\dfrac{\tau-c_1}{2}\tau'+c_1-1\right).
$$
This bundle is ample whenever
$$
\dfrac{\frac{\tau-c_1}{2}\tau'+c_1-1}{\tau'-2}>\dfrac{\tau-c_1}{2},
$$
that is, if $i>2$. Note that $i$ is strictly bigger than one, since Proposition \ref{prop:linesdiv} implies that $X$ is covered by rational curves of $-K_{X}$-degree equal to $i$. Since moreover $\deg(\pi_{|E})=\tau'=i'-2$ by equation (\ref{eq:E}), the result follows directly from Lemma \ref{lem:laz}, and the fact that the only $(n+1)$-dimensional manifold of index bigger than or equal to $n+2$ is $\P^{n+1}$. 
\end{proof}

We already have all the necessary ingredients to conclude the classification of Fano bundles of type (D).

\begin{prop}\label{prop:typeDfin}
Let $(X,\cE)$ be a pair of type {\rm (D)}. With the same notation as above, if $\pi_{|E}$ is finite, then $(X,\cE)$ is of type {\rm (D2)} or {\rm (D3)}.
\end{prop}

\begin{proof}
We claim first that $X'\cong\P^{n+1}$. Assuming the contrary, we have $i'\leq n+1$ and, by Lemma \ref{lem:lazapp}, $i=2$. In particular, by Lemma \ref{lem:muodd}, $c_1=-1$ and we may write $E=\tau' L+(\tau'-1)H$, so that $E$ corresponds to a nowhere vanishing section of the bundle $(S^{\tau'}\cE)(\tau'-1)$. Since the rank of this bundle is $\tau'+1=i'-1\leq n$, its top Chern class must then be equal to zero. The condition on the top Chern class reads as:
$$
\prod_{j=0}^{\tau'}\left((\tau'-j)\dfrac{-1+\sqrt{\Delta}}{2}+j\dfrac{-1-\sqrt{\Delta}}{2}+\tau'-1\right)=0.
$$
This means that $(\tau'-2j)\sqrt{\Delta}+\tau'-2=0$ for some $j$. Since $\Delta<0$ (see Lemma \ref{lem:MOS1}), it follows that
$\tau'=2$. In particular $E=2L+H=-K$ and $\rho=0$, so that, applying Equation (\ref{eq:rhotau}), we conclude:
$$
\arg\left(\sqrt{\Delta}(1+\sqrt{\Delta})^n\right)=\pi.
$$
Then $\sqrt{-\Delta}=\tan\left(\frac{\pi}{2n}\right)$ which is only possible for $n=2$ or $3$ (cf. \cite{Ni}). Since the first option contradicts $i=2$, and we have calculated that $i'=\tau'+2=4$, the only possibility is that $X'$ is a $4$-dimensional quadric, contradicting that $H^4(X',\Z)\cong\Z$.

Let us finally consider the case $X'\cong\P^{n+1}$.
Computing intersection numbers it is easy to see that a curve in $Y$ is a fiber of $\pi$ if and only if it is the strict transform in $Y$ of  an $n$-secant line of $T$. Hence through a general point of  $\P^{n+1}$ there is exactly one such line, and we may assert, in the language of \cite{DP}, that $X$ is the (irreducible) first order congruence of $n$-secants to $T$, and that $T$ is a codimension two subvariety of $\P^{n+1}$ with one apparent $n$-tuple point.
Such varieties are classified in \cite[Theorem 0.1]{DP}, and the only ones which have Picard number one are the twisted cubic $v_3(\P^1)\subset\P^3$ and the isomorphic projection of $v_2(\P^2)$ into $\P^4$. In other words, $(X,\cE)$ is of type (D2) or (D3).
\end{proof}


\section{Rank two Fano bundles with a conic bundle structure}\label{sec:proofconic}

In this section we will study pairs $(X,\cE)$ of type {\rm (C)}. 
Note that in this case we already know that $n$ equals $2$, $3$ or $5$ (cf. Lemma \ref{lem:rhotau}).  We will start by introducing some notation.

\begin{notn}\label{notn:C}
Following \cite{ABW}, the coherent sheaf $\cE':=\pi_*\big(\cO_{Y}(-K')\big)$ is locally free of rank three, and the natural map  $\pi'^*\cE'\to\cO_Y(-K')$ provides an embedding of $Y$ into the $\P^2$-bundle $p':P:=\P(\cE')\to X'$. We will denote with $L'$ the tautological line bundle of $\P(\cE')$ and the Chern classes of $\cE'$ by $c_1'$, $c_2'$ and $c_3'$; $c_1'\in H^2(X',\Z)$ will be identified with the corresponding integer; by abuse of notation the pull-back $p'^*H_{X'}$ will be denoted again by $H'$.
Denoting by $\cF$ the kernel of $\pi'^*\cE'\to\cO_Y(-K')$, the normal bundle of (the natural embedding of) $Y$ into $P\times_{X'}Y$ can be written as $\cN_{Y,P\times_{X'}Y}=\cF^\vee(-K')$. Since the relative tangent bundle of $P\times_{X'}Y$ over $P$ restricted to $Y$ is $-K'$, we conclude that:
$$
\cN_{Y,P}=\det(\cF^\vee(-K'))-(-K')=-2K'-\pi'^*\det(\cE').
$$
In particular it follows that
\begin{equation}\label{eq:Y}
Y\equiv_{\lin}2L'-c'_1H'.
\end{equation}
Finally, we will denote by $R\subset Y$ the closed subset of $Y$ where the rank of $d\pi'$ is not maximal. 
Abusing of notation we will denote its cohomology class by $R$, too. 
\end{notn}

It is known  that $\pi'(R)$ is a divisor with normal crossing singularities whose smooth points correspond to conics consisting of two distinct lines and whose singular points correspond to double lines (cf. \cite{S}).   
The following result shows some well known cohomological properties of $R$ that will be useful later on.

\begin{lem}\label{lem:conicforms}
With the same notation as above, the following equalities hold:
\begin{eqnarray}
&&R=-K'\pi'^*K_{X'}+c_{2}(T_Y)-\pi'^*c_{2}(T_{X'}),\label{eq:Porteous}\\
&&12c_1'H_{X'}={\pi'}_*\left(13K'^2+c_{2}(T_Y)\right)+2K_{X'},\label{eq:GRR}\\
&&{\pi'}_*R=-{\pi'}_*(K'^2)=-c_1'H_{X'},\label{eq:adj}\\
&&{\pi}_*R=\left((\nu'+2)(\nu\nu'-1)+2(\nu+1)\right)H_X-{\pi}_*{\pi'}^*c_{2}(T_{X'}).\label{eq:inX1}
\end{eqnarray}
\end{lem}

\begin{proof}
From the definition of $R$ as the degeneracy locus of $d\pi':T_Y\to T_{X'}$, its cohomology class may be computed using Porteous formula, that provides equation (\ref{eq:Porteous}).

On the other hand, Grothendieck-Riemann-Roch Theorem applied to $\pi'$ and $\cE'={\pi'}_*\cO_{Y}(-K')$ states that
$$
\ch(\cE')\td(T_{X'})={\pi'}_*\left(\ch(\cO_{Y}(-K'))\td(T_{Y})\right).
$$ 
The codimension one part of this equality  gives us:
\begin{eqnarray*}
-\frac{3}{2}K_{X'}+c_1'H_{X'}&=&{\pi'}_*\left(\frac{1}{2}K'^2+\frac{1}{2}K'K_Y+\frac{1}{12}(K_Y^2-c_{2}(T_Y))\right)\\
&=&{\pi'}_*\left(\frac{13}{12}K'^2+\frac{1}{12}c_{2}(Y)+\frac{2}{3}K'\pi'^*K_{X'}\right).
\end{eqnarray*}
Since ${\pi'}_*(K'\pi'^*K_{X'})=-2K_{X'}$, equation (\ref{eq:GRR}) follows.

The first equality of (\ref{eq:adj}) follows from \cite[Prop.~6]{Ch}: the author shows that $-\pi'_*(K_Y^2)=4K_{X'}+\pi'_*(R)$; substituting $K_Y$ by $\pi'^*(K_{X'})+K'$ we obtained the desired formula. The second follows from the combination of (\ref{eq:Porteous}) and (\ref{eq:GRR}). Finally, (\ref{eq:inX1}) follows from (\ref{eq:Porteous}) by projection formula and equation (\ref{eq:codim1}).
\end{proof}

The next proposition is based on Lemma \ref{lem:rhotau}. It shows that, as in the cases (P) and (D), $\mu$ is equal to one. As a by-product we obtain two useful numerical identities. 

\begin{prop}\label{prop:conicmu=1}
With the same notation as in Section \ref{sec:setup} let us further assume that $(X,\cE)$ is of type {\rm (C)}. Then:
\begin{eqnarray}
 &&\mu=1, \\
 &&c_{1}'=\dfrac{8}{\tau}\cos^2\left(\dfrac{\pi}{n+1}\right)-4\tau',\label{eq:c1E2}\\
 &&H_{X'}^n=\dfrac{\tau^{n-1}}{2^n\cos^{n-1}\left(\dfrac{\pi}{n+1}\right)}H_X^n.\label{eq:conicdegrees}
\end{eqnarray}
\end{prop}

\begin{proof}
We claim first that the following formula holds:
\begin{equation}\label{eq:CWconic}
K'(2K'+c_1'H')H'^{n-1}=0.
\end{equation}
In fact, we may consider the Chern-Wu relation for $\cE'$, whose intersection with $H'^{n-1}$ provides
$(L'^3-c'_1H'L'^2)H'^{n-1}=0$. Using (\ref{eq:Y}), this is equivalent to:
$$
L'\left(\left(\dfrac{Y+c_1'H'}{2}\right)^2-c_1'H'\dfrac{Y+c_1'H'}{2}\right)H'^{n-1}=0,
$$ 
which, since $H'^{n+1}=0$, may be written as $L'Y^2H'^{n-1}=0$. 
Restricting to $Y$, we finally obtain (\ref{eq:CWconic}).

We will also use the following equalities, that follow by using reduction modulo $(K^2-\Delta H^2)$ and the fact that $\arg\left(\tau+\sqrt{\Delta}\right)=\pi/(n+1)$ by Lemma \ref{lem:rhotau}:
\begin{equation}\label{eq:HH'}
H(-K+\tau H)^{n}=\dfrac{\tau^{n-1}}{\cos^{n-1}\left(\frac{\pi}{n+1}\right)}, \quad
H^2(-K+\tau H)^{n-1}=\dfrac{2\tau^{n-2}}{\cos^{n-3}\left(\frac{\pi}{n+1}\right)}.
\end{equation}
Since 
$$
-K'H'^{n}=\dfrac{-\nu'\mu^n}{2^{n+1}}\left(-K+\left(\tau-\dfrac{2}{\mu\nu'}\right)H\right)(-K+\tau H)^n=\dfrac{\mu^{n-1}}{2^{n}}H(-K+\tau H)^n,
$$
we obtain:
\begin{equation}\label{eq:K'}
-K'H'^{n}=\dfrac{(\mu\tau)^{n-1}}{2^n\cos^{n-1}\left(\frac{\pi}{n+1}\right)}(-KH^n).
\end{equation}
Analogously, we get
\begin{equation}\label{eq:K'2}
K'^2H'^{n-1}=\dfrac{\mu^{n-3}\tau^{n-2}}{2^{n-1}\cos^{n-1}\left(\frac{\pi}{n+1}\right)}\left(2\cos^2\left(\frac{\pi}{n+1}\right)-\nu'\mu\tau\right)(-KH^n).
\end{equation}

Using Equation (\ref{eq:CWconic}) together with (\ref{eq:K'}) and (\ref{eq:K'2}) we get:
\begin{equation}\label{eq:k}
c_1'=2\dfrac{K'^2H'^{n-1}}{-K'H'^{n}}=\dfrac{2}{\mu \nu}\left(4\cos^2\left(\dfrac{\pi}{n+1}\right)-2\nu\nu'\right).
\end{equation}
Denoting by $k$ the integer $4\cos^2\left(\frac{\pi}{n+1}\right)$, it follows that $2(k-4)\equiv 0$ mod $\mu$. 
For $n=3,5$ this and Lemma \ref{lem:muodd} imply that $\mu=1$. For $n=2$ we have necessarily $X \simeq X' \simeq \P^2$, so that $i=3$ and $\nu$ is even. Hence $\mu=1$ in any case and equation (\ref{eq:c1E2}) follows from (\ref{eq:k}).

Finally,
the third claimed equation follows from (\ref{eq:K'}) by substituting $\mu=1$ and noting that $2H_X^n=(-KH^n)$ and $2H_{X'}^n=(-K'H'^n)$.
\end{proof}

Note that, by construction, $\Pic(Y)\cong \Pic(P)=\Z(L',H')$ and we may compare the nef cones of $Y$ and $P$ inside of $N^1(Y)_{\R}\cong N^1(P)_{\R}$:

\begin{lem}\label{lem:cones}
Let $(X,\cE)$ be a pair of type {\rm (C)}. With the same notation as above, $\Nef(P)\subseteq\Nef(Y)$ and equality holds unless $Y\cdot f>0$. In particular, if  $Y\cdot f\in\{0, -1\}$, then $P$ is a Fano manifold.
\end{lem}

\begin{proof}
The obvious inclusion $\Nef(P)\subseteq\Nef(Y)$ is strict if and only if there exists an irreducible curve $C \subset P$, not contained in $Y$, satisfying $C\cdot(L'+\tau' H')<0$. By equation (\ref{eq:Y}) this is equivalent to:
$$
0\leq C\cdot Y<-(c_1'+2\tau')C\cdot H'.
$$ 
In particular, using (\ref{eq:Y}) again, together with Proposition \ref{prop:conicmu=1}, it follows that  $Y\cdot f=-(c_1'\mu+2\tau')=-(c_1'+2\tau) >0$.
For the second part, note that if $\Nef(P)=\Nef(Y)$, then $P$ is Fano if and only if the anticanonical degree of $f$ and $f'$ is negative. Then the result follows by adjunction.
\end{proof}

\subsection{Determining the invariants of $(X,\cE)$}\label{ssec:inv}

Using Proposition \ref{prop:conicmu=1}, Lemma \ref{lem:rhotau}, Lemma \ref{lem:muodd} and the fact that $\tau'>0$, 
 we get the following possibilities for the cases $n=2,3$: \\
 
\begin{table*}[!!h]
\centering
\begin{tabular}{|c|c|c|c|c|c|c|c|c|c|c|c|c|c|c|}\hline
$n$&$\tau$&$i$&$d$&$H_X^n$&$\tau'$&$i'$&$H_{X'}^n$&$c_1$&$\Delta$&$c_{2}$&$X$&$X'$&$c_1'$&$Y\cdot f$\\\hline\hline
$2$&$2$&$3$&$1$&$1$&$1$&$3$&$1$&$0$&$-12$&$3$&$\P^2$&$\P^2$&$-3$&$1$\\\hline
\multirow{2}{*}{$3$}&$1$&$2$&$4$&$4$&$2$&$4$&$1$&$-1$&$-1$&$2$&$V_4^3$&$\P^3$&$-4$&$0$\\\cline{2-15}
                    &$2$&$3$&$2$&$2$&$1$&$3$&$2$&$0$ &$-4$&$2$&$\Q_3$&$\Q^3$&$-2$&$0$\\\hline                                      
\end{tabular}
\end{table*}

The determination of the numerical invariants of $(X,\cE)$ in the case of fivefolds is more involved. Along the rest of the section we will show that the only possible values for $(\tau,\tau')$ are $(1,3)$ and $(3,1)$, so that, using Proposition \ref{prop:conicmu=1}, Lemma \ref{lem:rhotau} and Lemma \ref{lem:muodd} one easily gets the following possibilities: 
\begin{table*}[!h]
\centering
\begin{tabular}{|c|c|c|c|c|c|c|c|c|c|c|c|c|}\hline
$n$&$\tau$&$H_X^n$&$\tau'$&$H_{X'}^n$&$c_1$&$\Delta$&$c_{2}$&$X$&$X'$&$c_1'$&$Y\cdot f$\\\hline\hline
\multirow{2}{*}{$5$}&$3$&$4$&$1$&$18$&$-1$&$-3$&$d$&$V_4^5$&$K(\Gg_2)$&$-2$&$0$\\\cline{2-12}
&{$1$}&$36$&$3$&$2$&{$-1$}&{$-1/3$}&{$d/3$}&$W_{36}^5$&$\Q^5$&$-6$&$0$\\\hline
\end{tabular}
\end{table*}

\begin{lem}\label{lem:5fdet1}
Let $(X,\cE)$ be a pair of type {\rm (C)}. With the same notation as above, if $n=5$, then 
$$
(\tau,\tau')\in\left\{(1,3),(1,4),(2,1),(3,1)\right\}.
$$ 
\end{lem}
\begin{proof}
Note first that (\ref{eq:c1E2}) tells us that $6/\tau=c_1'+4\tau'\in\Z$, hence $\tau\in\{1,2,3\}$. Moreover we know that $1\leq \tau'=i'-2$ hence the the classification of Fano manifolds of coindex $\leq 3$ (cf. \cite{Fu}, \cite{Mu}) provides a finite number of possible values of $H_{X'}^5$. Comparing this list with equation (\ref{eq:conicdegrees}), that for $n=5$ says $$H_{X}^5=\frac{18}{\tau^4}H_{X'}^5,$$ one easily gets that if $\tau$ is equal to $2$ or $3$, then $\tau'=1$.

Assume now that $\tau=1$, and let us show that in this case $X'$ is $\P^5$ or $\Q^5$, that is $\tau'\geq 3$.
If $\tau'=1$, then (\ref{eq:adj}) and (\ref{eq:c1E2}) provide ${\pi'}_*R=-c_1'=-2$, a contradiction. If $\tau'=2$ then  $X'$ is a del Pezzo fivefold, whose classification allows us to compute the list of possible values of $c_{2}(X')$. Plugging them into equation (\ref{eq:inX1}) one gets:
$$
{\pi}_*R=\left\{\begin{array}{ll}-9H_X&\mbox{if }H_{X'}^5=1\\-3H_X&\mbox{if }H_{X'}^5=2\\-H_X&\mbox{if }H_{X'}^5=3\\0&\mbox{if }H_{X'}^5=4\end{array}\right.
$$
Note that we have not considered the del Pezzo fivefold of degree $5$ because its fourth Betti number equals $2$.
Since $\pi_*R$ is effective, we may exclude the first three cases. In the fourth case, ${\pi}_*R=0$ tells us that $\pi(R)$ has codimension bigger than one. But $\tau=1$ implies that $\cE(1)$ is nef and $c_1=-1$ (cf. Lemma \ref{lem:muodd}), hence $\cE$ is uniform on lines, with splitting type $(0,-1)$. In particular there is a section of $\P(\cE)$ over the general line $\ell$ contracted by $\pi'$, hence meeting $R$,
a contradiction. 
\end{proof}

\begin{lem}\label{lem:5fdet2}
With the same notation as in Lemma \ref{lem:5fdet1}, $(\tau,\tau')\neq(2,1).$ 
\end{lem}

\begin{proof} Assume the contrary. Then by (\ref{eq:conicdegrees}) $X$ and $X'$ are two Fano manifolds of Picard number one and coindex three whose degrees satisfy $H_{X}^5=\frac{9}{8}H_{X'}^5$. Using \cite{Mu}, one finds that the only possibility is $H_{X}^5=18$, $H_{X'}^5=16$, that is $X$ is $K(G_2)$ and $X'$ is a hyperplane section of the Lagrangian Grassmannian $L\G(3,6)$.

Moreover $c_1'=-1$ by (\ref{eq:c1E2}), so that one may compute $Y\cdot f=-1<0$ and Lemma \ref{lem:cones} implies that $\pi$ extends to an elementary contraction $p:P\to Z$.
Since the exceptional locus of $p$ is $Y$ and all the exceptional fibers of $p$ are $\P^1$'s, it follows that $p$ is the smooth blow-up of $Z$ along $X$. In particular, since $\Pic(P)\to\Pic(Y)$ is an isomorphism, the restriction map $\Pic(Z)\to\Pic(X)$ is an isomorphism, too. As usual, let us denote by $H_Z$ the ample generator of $\Pic(Z)$.
Since moreover $c_1=0$ and $Y_{|Y}=-2K'+H'=K/2+H=L+H$ --by Lemma \ref{lem:muodd} and equations (\ref{eq:codim1}), (\ref{eq:Y})-- it follows that the normal bundle to $X$ in $Z$ is $\cE(1)$ and, by adjunction, ${-(K_{Z})}_{|X}=5H_{X}$. We may then assert that $-K_Z=5H_Z$ and that, in particular, the degree of $Z$ must be smaller than or equal to $22$ (cf. \cite{Mu}). 

Note also that, arguing as in  \ref{notn:D}, since $H^4(X',\Z)\cong\Z$ by Lefschetz Theorem, necessarily $H^4(Z,\Z)\cong\Z$ as well. Let us denote by $\Sigma_Z$ the positive generator of $H^4(Z,\Z)$, and set $X=m\Sigma_Z$, $H_Z^2=d_{Z}\Sigma_Z$. From the computations
$$
22\geq H_Z^7=d_{Z}^3H_Z\Sigma^3,\quad\mbox{and}\quad 18=H_Z^5X=\frac{m}{d_{Z}}H_Z^7=d_{Z}^2m H_Z\Sigma^3,
$$ 
one may conclude $d_{Z}=1$, 
so that, in particular $18=mH_Z^7$. But on the other hand, since $X^2=c_{2}(\cE(1))=\frac{4}{3}H_{X}^2$, we get
$m^2H_Z^7=X^2H_Z^3=\frac{4}{3}H_Z^5X=24,$ a contradiction.
\end{proof}

\begin{lem}\label{lem:5fdet3}
With the same notation as in Lemma \ref{lem:5fdet1}, $(\tau,\tau')\neq(1,4).$ 
\end{lem}

\begin{proof}
If $(\tau,\tau')=(1,4)$ then $X'=\P^5$ and $X$ is a Fano manifold of index $3$ and degree $36$ by (\ref{eq:conicdegrees}). In particular we may identify all the Chern classes of $\cE'$ with integers, that we will denote by $c_1'$, $c_2'$ and $c_3'$ as well.

In order to exclude this case, it is enough to show that $c_3'$ is odd; in fact, since $c_1'=-10$ is even --by (\ref{eq:c1E2})--, this implies that $\cE'$ does not satisfy Schwarzenberger condition $c_1'c_{2}'\equiv c_3'$ (mod $2$). 
Moreover, using recursively the Chern-Wu relation of $\cE'$ we get 
$$L'^5H'^2=c_1'^3-2c_1'c_2'+c_3',$$
so that it suffices to show that $L'^5H'^2$ is odd. 

Using that $Y=2L'-c_1'H'$, we may write:
\begin{eqnarray}\label{eq:L5}
L'^5H'^2&=&\dfrac{1}{2}L'^4H'^2Y+\dfrac{c_1'}{4}L'^3H'^3Y+\dfrac{c_1'^2}{8}L'^2H'^4Y+\dfrac{c_1'^3}{16}L'H'^5Y\\
&=&\dfrac{1}{2}K'^4H'^2-\dfrac{c_1'}{4}K'^3H'^3+\dfrac{c_1'^2}{8}K'^2H'^4-\dfrac{c_1'^3}{16}K'H'^5=395
\end{eqnarray}

Note that, by Lemma \ref{lem:rhotau}, $\Delta=-\tau^2\tan^2\left(\frac{\pi}{6}\right)=-1/3$, so that we have an equality $K^2=-H^2/3$, that in terms of $K'$ and $H'$ reads as $K'^2-5K'H'+7H'^2=0$. Then reducing (\ref{eq:L5}) modulo this polynomial one finally gets:
$$
L'^5H'^2=\dfrac{395}{2}K'H'^5=-395.
$$
\end{proof}

\subsection{Proof of Theorem \ref{thm:fanobdl} {\rm (C)}}

In this section we will complete the proof of Theorem \ref{thm:fanobdl}, by showing that the only pairs $(X,\cE)$ of type {\rm (C)}, whose possible numerical invariants have been computed in Section \ref{ssec:inv},  are the pairs of types {\rm (C1-6)}. In order to see this, we will consider the universal family $q:\cU\to \cM$ of curves of $H_{X'}$-degree one in $X'$. By Proposition \ref{prop:conicmu=1}, there exists a morphism $\psi:X\to \cM$ such that $\pi:Y\to X$ is the pull-back of $q$: 
$$
\xymatrix{X\ar@{->}[d]&Y\ar@{->}[d]\ar[l]_{\pi\hspace{0cm}}\ar[r]^{\hspace{0cm}\pi'}\ar@{}[l]^{\hspace{0.8cm}{\text{\tiny$\square$}}}&X'\\
\cM&\cU\ar[l]_{q}\ar[r]&X'\ar@{=}[u]\\}
$$

Note that, according to Section \ref{ssec:inv}, in each case $\cM$ is a smooth Fano manifold of Picard number one. Let us denote by $H_\cM$ the ample generator of $\Pic(\cM)$ and by $\tau_\cM$ the positive number such that  $-K_{\cU|\cM}+\tau_\cM q^*H_\cM$ is nef and not ample on $\cU$, whose computation in each case is straightforward. 

\begin{table*}[!h]
\centering
\begin{tabular}{|c|c|c|c|c|}\hline
$X'$&$\cM$&$\tau_\cM$&$X$&$\tau$\\\hline\hline
$\P^2$&$\P^2$&$2$&$\P^2$&$1$\\\hline
$\P^3$&$\G(1,3)$&$1$&$V_4^3$&$1$\\\hline
$\Q^3$&$\P^3$&$2$&$\Q^3$&$2$\\\hline 
$K(\Gg_2)$&$\Q^5$&$3$&$V_4^5$&$3$\\\hline                   
$\Q^5$&$\G(1,6)_{\Q^5}$&$1$&$W_{36}^5$&$1$\\\hline                                        
\end{tabular}
\end{table*}

In particular, since the pull-back of $-K_{\cU|\cM}+\tau_\cM q^*H_\cM$ to $Y$ is $-K+\tau H$, we may write \begin{equation}\label{eq:psiH}
\psi^*H_\cM=\frac{\tau_\cM}{\tau}H_{X},
\end{equation} and conclude the following:

\begin{prop}\label{prop:conicproof1}
Let $(X,\cE)$ be a pair of type {\rm (C)}. With the same notation as above, if $X$ is $\P^2$, $\Q^3$ or $V_4^5$, then $(X,\cE)$ is of type {\rm (C1)}, {\rm (C3)}, or {\rm (C4)}, respectively. If $X$ is $V_4^3$, then $(X,\cE)$ is of type {\rm (C2)} or {\rm (C6)}.
\end{prop}

\begin{proof}
In the case $X=\P^2$, from the previous discussion we get $\psi^*H_{\P^2}=2H_{\P^2}$ so that $\psi$ is given by a base point free two dimensional linear system on $\P^2$ of degree two, i.e. $(X,\cE)$ is of type {\rm (C1)}. 

If $X$ is $\Q^3$ or $V_4^5$, then  $\psi^*H_\cM=H_X$ is very ample, so $\psi$ consists of a linear projection of the associated embedding by the complete linear system $|H_X|$. If $X$ is $\Q^3$, this already tells us that $\psi:\Q^3\to \P^3$ is necessarily the $2$ to $1$ morphism given by a projection from any outer point; this construction provides a pair of type {\rm (C3)}. If $X=V_4^5$, then $\psi:X\subset\P^7\to\Q^5\subset\P^6$ must be the projection from one of the (eight) points that are vertices of the singular quadrics containing $X$, and $(X,\cE)$ is of type {\rm (C4)}. 

If $X=V_4^3$, we have two possibilities for the morphism $\psi:X\subset\P^5\to\G(1,3)\subset\P^5$.  If $\psi$ is given by the complete linear system $|H_X|$, then $V_4^3$ is the complete intersection of $\G(1,3)$ with another quadric, and $(X,\cE)$ is of type {\rm (C6)}. If this is not the case, then $\psi(X)$ lies in a linear section of $\G(1,3)$. Since the degree of these sections is two, necessarily $\psi$ is $2$ to $1$ onto a $3$-dimensional quadric; the center of projection is, again, one of the (six) points that are vertices of the singular quadrics containing $V_4^3$. Note that $\psi(X)$ is smooth since this $X=V_4^3$ is not contained in quadrics of rank smaller than $5$ (see \cite[Proposition 2.1]{R}). 
\end{proof}

Finally we will consider the case $X=W_{36}^5$, $X'=\Q^5$. Here our line of argumentation goes through the analysis of  $p':P\to X'$. Note that we have seen in Section \ref{ssec:inv} that in this case $Y\cdot f=0$, hence Lemma \ref{lem:cones} tells us that $P$ is a Fano manifold and that $\pi:Y\to X$ extends to a contraction $p:P\to Z$. Note also that, since $Y\cdot f=0$ implies that the restriction of $T_{P}$ to $f$ is nef, $p$ is of fiber type. 

\begin{prop}\label{prop:conicproof2}
Let $(X,\cE)$ be a pair of type {\rm (C)}, with $X=W_{36}^5$. Then $(X,\cE)$ is of type {\rm (C5)}.
\end{prop}

\begin{proof}
We begin by proving that $p:P\to Z$ has at most a finite number of fibers that are different from reduced $\P^1$'s. 
In fact, since $p$ is elementary and $Y\cdot f=0$ it follows that fibers of $p$ meeting $Y$ must be contained in $Y$. In particular the general fiber of $\pi$ is one dimensional and, since $p^*H_{X'}\cdot f=1$, every fiber of dimension one is irreducible and reduced. 
On the other hand, $Y\cdot f=0$ tells us that $Y$ is nef on $P$ and that it has positive intersection number with curves not contracted by $p$, so that we may assert that the set $J$ of points in $Z$ whose inverse image have dimension greater than one is discrete. 

Note that, in particular, $\psi:X\to \cM=\G(1,6)_{\Q^5}$ extends to a morphism $Z\setminus J\to\cM$, that we denote by $\psi_\cM$.

We claim that there exists a point $z_0\in J$ such that $p^{-1}(z_0)$ dominates $X'=\Q^5$.
Assume the contrary; then $p'^{-1}(x)\cap p^{-1}(J)=\emptyset$ for the general point $x$, and we have a non constant morphism 
$$(\psi_\cM\circ p)_{|p'^{-1}(x)}:p'^{-1}(x)\cong\P^2\to\cM.$$ whose image is contained in the family of lines in $X'$ passing by $x$, which is isomorphic to $\Q^3$. On the other hand, by (\ref{eq:psiH}) and Proposition \ref{prop:conicmu=1} we have
$$
(\pi^*\psi^*H_\cM)\cdot f'=\pi^*H_X\cdot f'=H\cdot f'=1,
$$ so $(\psi_\cM\circ p)_{|p'^{-1}(x)}^*H_\cM=H_{\P^2}$, 
contradicting the fact that $\Q^3$ does not contain planes.

We will study now the splitting type of $\cE'$ on lines in $X'$. Note first that, since $\tau'=3$ and $c_1'=-6$, (see Section \ref{ssec:inv}), $\cE'(3)$ is a nef (and not ample) bundle of degree $3$. Given any line $\ell\in X'$ the previous claim tells us that $\P(\cE'(3)_{|\ell})$ meets $F_0:=p^{-1}(z_0)$ along a curve that is contracted by $p$, i.e. that $\cE'(3)_{|\ell}$ has a trivial summand. If $\ell\in\psi(X)$, there is a section of $\P(\cE'(3))$ over $\ell$ contained in $Y$ and contracted by $p$. Since $F_0\cap Y=\emptyset$, this section is different from the one contained in $\P(\cE'(3))\cap F_0$ and, as $c_1(\cE'(3))=3$, it follows that $\cE'(3)_{\ell}=\cO_{\ell}^{\oplus 2}\oplus\cO_{\ell}(3)$. Conversely if $\ell$ is a line for which $\cE'(3)$ has splitting
$(0,0,3)$, then $Y$ meets the locus of minimal sections of $\P(\cE'(3)_{|\ell})$, hence $\ell\in\psi(X)$. 
Furthermore, since $\cE'(3)$ is nef, the splitting type of $\cE'(3)$ on a line of $\cM\setminus\psi(X)$ must be $(0,1,2)$. 

Now we will prove that $\psi(X)\subset\cM$ is isomorphic to the natural embedding of 
$K(\Gg_2)$ into $\cM=\G(1,6)_{\Q^5}$. Let $F_0'$ be an irreducible component of $F_0$ dominating $\Q^5$. By the previous arguments $F_0'$ meets the general $p'^{-1}(x)$ precisely in one point. Since $\Q^5$ is normal, $F_0'$ is a section of $\P(\cE')$, given by a quotient $\cE'(3)\to\cO_{\Q^5}$. Denote by $\cF$ its kernel.  
By \cite[Proposition 1.2]{CP} $\cF$ is a nef rank two bundle on $\Q^5$, with first Chern class equal to $3$. Then $\cF$ is a Fano bundle, of type {\rm (P)} by Lemma \ref{lem:muodd} and so, by Proposition \ref{prop:typeP}, it is a twist of a Cayley bundle $\cC$ on $\Q^5$ with $\cO_{\Q^5}(2)$. Since $H^1(\Q^5,\cC(2))=0$ (cf. \cite[Theorem 3.1]{O}), it follows that \begin{equation}\label{eq:cayley}
\cE'(3)\cong\cO_{\Q^5}\oplus\cC(2),
\end{equation}
and $\psi(X)$ may be described as the set of jumping lines of $\cC(2)$, which is $K(\Gg_2)$. 

Finally, the isomorphism (\ref{eq:cayley}) tells us that $Z$ is a cone with vertex a point over $\psi(X)=K(\Gg_2)$, and that $p^*H_Z=L'+3H'$, where $H_Z$ is the extension to $Z$ of the line bundle $(H_\cM)_{|\psi(X)}$. Then, since $Y$ is linearly equivalent to $2L'+6H'=2 p^*H_Z$, it follows that $X$ is a quadric section of $Z$, i.e. that $(X,\cE)$ is of type {\rm (C5)}. 
\end{proof}

\appendix
\section{Fano bundles and congruences of lines of order one}\label{sec:appendix}

Congruences of lines have been an object of study by algebraic geometers since the nineteeth century. In this appendix we will consider them in relation to our work on rank two vector bundles, since some of them provide interesting examples of Fano bundles (see Lemma \ref{lem:covered} below). One of their most remarkable properties is that, under certain hypotheses, their varieties of minimal rational tangents (VMRT for short) might be linear and reducible (Proposition \ref{prop:VMRT}). We present in Proposition \ref{prop:severi} an example of this particular type of congruences, giving a negative answer to a problem posed by Hwang and Mok on the irreducibility and non linearity of the VMRT.

\begin{defn}\label{defn:cong}
A subvariety $X \subset \G(1,m)$ of dimension $m-1$ is called a {\it congruence of lines in $\P^m$}. 
Denoting by $Y$ the projectivization $\P(\cQ_{|X})$ of the restriction to $X$ of the universal quotient bundle $\cQ$ and by $\pi$ and $\pi'$ the corresponding projections $\pi:Y \to X$ and $\pi': Y \to \P^m$, we define the {\it order} of $X$ as the number of lines parametrized by $X$ passing through a general point of $\P^m$, which is zero if $\pi'$ is of fiber type or equal to the degree of $\pi'$ otherwise.
A point $y \in \P^m$ is called {\it fundamental} if the fiber $\pi'^{-1}(y)$ has dimension greater than $m-\dim(\pi'(Y))$, and the {\it fundamental locus} $Z$ will be the set of the fundamental points. Finally, we will say that a congruence {\it has linear fibers} if
the images of the fibers of $\pi'$ are linear spaces in $\G(1,m)$.
\end{defn}

\begin{ass}\label{ass:setupA}
Throughout this appendix, $X$ will denote a smooth congruence of lines of order one in $\P^m$ satisfying $\Pic(X) \cong \Z$. With the same notation as above, we will denote by $Z\subset\P^m$ its fundamental locus, and by $E$ the exceptional locus of $\pi':Y=\P(\cQ_{|X})\to\P^m$, i.e. the inverse image of $Z$. Since $\pi'$ is birational between smooth varieties, then $E$ is a divisor. As usual, we denote by $L$ the divisor associated to the tautological divisor of $\P(\cQ_{|X})$ and by $H$ the pull-back to $Y$ of the ample generator of $\Pic(X)$.
\end{ass}

The next result shows that congruences satisfying these assumptions provide examples of Fano bundles.

\begin{lem}\label{lem:covered} Let $X$ be a congruence of lines satisfying Assumptions \ref{ass:setupA}.  Then $X$ and $Y$ are Fano manifolds. If moreover $X$ has linear fibers, then $X$ is covered by lines and, in particular, $\Pic(X)$ is generated by the restriction to $X$ of the Pl\"ucker line bundle in $\G(1,m)$.
\end{lem}
 \begin{proof}
We will prove that $-K_Y$ is positive on fibers of $\pi'$, since then $Y$ will be a Fano manifold
by the Kleiman criterion, and so also $X$ will be Fano by \cite[Theorem 1.6]{SW2}.

The intersection of the exceptional divisor $E$ with the strict transform $\ell$ of a line in $\P^m$ is zero. Then since the class of $\ell$ lies in the interior of $\cNE{Y}$ (it is not contracted neither by $\pi$ nor by $\pi'$) and the class of $E$ is of the form $\alpha L+\beta H$, $\alpha\geq 0$, it follows that $E$ is negative on curves contracted by $\pi'$ and $\alpha>0$. Moreover $-K_{Y} = -\pi'^*K_{\P^m}-cE$ with $c >0$, and the claim follows.

To show the last assertion it is enough to note that $E$ dominates $X$, and this holds because $\alpha>0$.
 \end{proof}

The next results describes the second contraction of the universal family of lines on a congruence satisfying \ref{ass:setupA}.

\begin{lem}\label{lem:blowup} Let $X$ be a congruence of lines satisfying Assumptions \ref{ass:setupA}. Then its fundamental locus $Z$ is irreducible and $\pi':Y \to \P^m$ is the normalization of the blow up of $\P^m$ along $Z$.
\end{lem}

\begin{proof} As we have seen in Lemma \ref{lem:covered}, $\pi'$ is a Fano-Mori contraction with exceptional locus $E$. Moreover since $\Pic(X)=\Z$, $\pi'$ is elementary, hence $E$ and its image $Z \subset \P^m$ (which is the fundamental locus of $X$) are irreducible. By the universal property of the blow up $b:B\to\P^m$ of $\P^m$ along $Z$, $\pi'$ factorizes through a surjective morphism $\f:Y\to B$, that is finite because the Picard number of $B$ is not one. Since $\f$ is generically one to one, it factors via an isomorphism from $Y$ to the normalization of $B$.
\end{proof}

Given a smooth variety $X$ and a family of rational curves parametrized
by $\cM$ such that $\cM_x$ is proper for the
general point $x \in X$, the {\it variety of minimal rational tangents} (VMRT
for short) of $X$ at $x$ with respect to $\cM$ is the closure of the set of points in $\P(\Omega_{X,x})$
corresponding to the tangent lines of the general curves of the family
$\cM$ passing by $x$. We refer to \cite{Hw} for a complete account on
the VMRT. The next result shows that, at a general point, the VMRT of a congruence of Picard number one, order one and with linear fibers is a disjoint union of linear spaces.

\begin{prop} \label{prop:VMRT} Let $X$ be a congruence of lines satisfying Assumptions \ref{ass:setupA} and having linear fibers. Let $z:=dim(Z)$ and set $\alpha:= E \cdot f\in \Z$, where $f$ denotes the cohomology class of a fiber of $\pi:Y\to X$. Then $\alpha=(m-1)/(m-z-1)$, $X$ has index $m-z$ and the VMRT of $X$ at a general point is the disjoint union of $\alpha$ linear spaces of dimension $m-z-2$.
\end{prop}

\begin{proof}
Denote by $f'$ the class of a line in a fiber of $\pi'$, which, by hypothesis, is a minimal section of $Y$ over a line in $X$.

By Lemma \ref{lem:blowup}, outside of the set of singular points of $Z$, $\pi'$ is the smooth blowup of $Z$, hence $E \cdot f'=-1$ and $K_{Y} = \pi'^*K_{\P^{m}} + (m-z-1)E$. In fact $K_Y-\pi'^*K_{\P^m}$ is proportional to $E$, and their ratio may be computed outside of the singular points of $Z$. Since $L \cdot f'=0$, $H \cdot f'=1$, $L\cdot f=1$, $H\cdot f=0$,
we can  write
\begin{equation} \label{eq:Ered}
E= \alpha L - H
\end{equation}
from which it follows that $\alpha=(K_Y\cdot f+m+1)/(m-z-1)=(m-1)/(m-z-1)$ and $i_X=m-z$.

To finish the proof we observe that, given any line $\ell$ in $X$, $\pi'$ contracts the minimal section of $\P(\cQ|_\ell)$ over $\ell$, hence, given a general point $x \in X$ the lines passing through $x$ are contained in the $\alpha$ linear subspaces passing through $x$ which are images of
the $\alpha$ fibers of $\pi'$ meeting $\pi^{-1}(x)$. This spaces have dimension $m-z-1$ by Lemma \ref{lem:blowup}.
\end{proof}

The next proposition shows that the fundamental loci of a congruence with linear fibers satisfying Assumption \ref{ass:setupA} cannot be a complete intersection in $\P^m$.

\begin{prop} \label{prop:degZ}
Let $X$ be as in Prop. \ref{prop:VMRT}, and assume $z>0$. Then $\deg(Z) < \alpha^{m-z}$ and, in particular, $Z$ cannot be a complete intersection.
\end{prop}

\begin{proof}
For $k>z$ we have in $Y$ that $L^{k}EH^{m-k-1} =0,$ therefore, using (\ref{eq:Ered}), we get
\begin{equation}\label{eq:Zrecur}
L^kH^{m-k}=\alpha L^{k+1}H^{m-k-1},\mbox{ for every }k>z.
\end{equation}
On the other hand the degree of $Z$ is given by
$$\deg (Z) = L^{z}EH^{m-z-1} = \alpha  L^{z+1}H^{m-z-1} - L^{z}H^{m-z},$$
hence, since $L^{z}H^{m-z}>0$ (because $z>0$) and $L^{m}=1$, using recursively (\ref{eq:Zrecur}) we get $$\deg (Z) < \alpha  L^{z+1}H^{m-z-1} = \alpha^{m-z}.$$

For the second part, note that if $Z$ were a complete intersection, it would be contained in a hypersurface of degree smaller than $\alpha$. This  contradicts the fact that for a general point of $\P^m$ there is an $\alpha$-secant line of $Z$.
\end{proof}

Let us concentrate in the case $Z$ smooth in order to use the examples appearing in the different classifications of smooth varieties with one apparent multiple point (see for instance \cite{DP} and the references therein).
If moreover $z=\dim Z > 2m/3$, $m\geq 7$, then $Z$ is conjectured to be a complete intersection. Hence the natural range to look for possible fundamental loci is $z \leq 2m/3$, where we have the following:

\begin{cor} With the same notation as above, assume that $Z$ is smooth and that $0<z \leq 2m/3$. 
Then the possible values of $(\alpha,z,m)$ are:
$$
(3,2k,3k+1)
\mbox{ with $k>0$ },\,\, (4,3,5),\,\, (4,6,9) \mbox{ and }\,\, (5,4,6).
$$
\end{cor}

\begin{proof} Let us first show that $\alpha>2$. If $\alpha=2$ then, by Prop. \ref{prop:VMRT}, $X$ is a Fano manifold of Picard number one, covered by linear spaces of dimension $m-z-1 =\dim X/2$. This implies that $X$ is a projective space, a quadric or a Grassmannian of lines by \cite[Main Theorem]{Sa}, \cite[Corollary 5.3]{NO}. On the other hand, by Prop. \ref{prop:VMRT} again, the index of $X$ is $m-z=\dim X/2+1$, then the only possibility is $X=\Q^2$, whose Picard number is not one.
Now a simple computation shows that the only possible cases are $(\alpha, z, m) = (3,2k, 3k+1),(4,3,5),(4,6,9), (5,4,6)$.
\end{proof}

We will discuss now the existence of examples of congruences satisfying the properties that we have imposed in this appendix, paying attention to their VMRT's. For the types $(4,3,5)$ and $(5,4,6)$ the VMRT at the general point is a finite set of points; an example of type $(4,3,5)$ appears by considering the congruence of $4$-secant lines to a Palatini threefold in $\P^5$ (see \cite[Thm.0.1]{DP}), but there are not known examples of type $(5,4,6)$ with $Z$ smooth. To our best knowledge examples of type $(4,6,9)$, for which the VMRT would be a disjoint union of lines, are not known either.

Examples of type $(3,2k, 3k+1)$ appear by considering congruences of trisecant lines of smooth projections of Severi varieties.
Let us recall that a {\it Severi variety} is a smooth $2k$-dimensional projective variety $S \subset \P^{{3k}+2}$ which can be isomorphically projected to $\P^{3k+1}$. They have been classified by Zak (cf. \cite[Chapter 4]{Z}), who proved, in particular, that the only possible values of $k$ are $1,2,4$ and $8$. The last part of this appendix is devoted to showing the following

\begin{prop}\label{prop:severi}
Let $X\subset\G(1,3k+1)$ be the closure of the family of trisecant lines to a general isomorphic projection $Z\subset\P^{3k+1}$ of a $2k$-dimensional Severi variety $S\subset\P^{3k+2}$. Then $X$ is a smooth congruence of order one with linear fibers.
\end{prop}

\begin{rem}\label{rem:hmconj}
In particular Proposition \ref{prop:VMRT} tells us that the VMRT at a general point of $X$ consists of the disjoint union of three linear spaces of dimension $k-1$. Hence this gives (for $k=2,4,8$) a negative answer to the problems of irreducibility (cf. \cite[Question 2]{Hw}) and non linearity (\cite[Conjecture, p.52]{HM}) of the VMRT of a Fano manifold of Picard number one. Note also that Hwang has shown that the number of irreducible components of the VMRT at a general point is at least three (cf. \cite[Prop.~2]{Hw2}).
\end{rem}

In \cite{IM} the authors showed, using techniques of Representation Theory and Jordan Algebras, that the congruences of Proposition \ref{prop:severi} are in fact linear. We will show how \ref{prop:severi} may be obtained by using the following result due to Ein and Sheperd-Barron (see \cite[Thm.~2.6]{E-SB}):

\begin{lem}\label{lem:esb}
Let $S\subset\P^{3k+2}$ be a $2k$-dimensional Severi variety. The linear system of quadrics containing $S$ provides an involutive Cremona transformation $\psi:\P^{3k+2}\to\P^{3k+2}$, fitting in a diagram:
$$
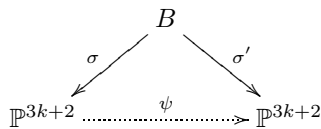
\xymatrix{&B\ar[ld]_{\sigma}\ar[rd]^{\sigma'}&\\\P^{3k+2}\ar@{..>}[rr]^{\psi}&&\P^{3k+2}}
$$
where $\sigma$ and $\sigma'$ denote, respectively, the blowing up of $\P^{3k+2}$ along $S$ and the blowing up of $\P^{3k+2}$ along $S'\cong S$.
\end{lem}

\begin{proof}[Proof of Proposition \ref{prop:severi}]
As usual we denote by $\pi:Y\to X$ the universal family of lines over $X$, and by $\pi'$ the evaluation morphism to $\P^{3k+1}$. Let $O$ be the center of the projection from $\P^{3k+2}$ to $\P^{3k+1}$, which is a general point in $\P^{3k+2}$.

We begin by showing that $X$ is a congruence of order one. In fact we will show that for every $P \in \P^{3k+1}\setminus Z$ passes a unique trisecant line to $Z$. This is equivalent to prove that there is a unique plane $\Lambda$ in $\P^{3k+2}$ containing the line $\langle O,P\rangle$ which is trisecant to $S$.
Since, by construction, the line $\langle O,P\rangle$  does not meet $S$, then the strict transform of $\langle O,P\rangle$ via $\psi$ is a conic $C$. Since $\psi$ is involutive, $C$ must meet $S'$ in three points, and the strict transform via $\psi$ of the plane $\Lambda'$ containing this conic is a plane $\Lambda$, containing $\langle O,P\rangle$, which is trisecant to $S$.

Set $E:=\pi'^{-1}(Z)$. We will prove that $\pi(\pi'^{-1}(P))$ is a $k$-dimensional linear space for all $P\in Z$. From this it follows that $E\subset Y$ is an irreducible divisor and, by the universal property of the blowing up, the morphism $\pi'$ factors via a finite morphism from $Y$ to the blow up $\widetilde{\P}^{3k+1}_Z$. Since we have already seen that $\pi'$ is birational, then we may assert that $Y$ is in fact isomorphic to $\widetilde{\P}^{3k+1}_Z$. In particular $Y$ is smooth of Picard number two, and  we may conclude that $X$ is smooth of Picard number one.

Let $P\in Z$ be any point and consider the line $\langle O,P\rangle$. The strict transform of $\langle O,P\rangle$ by $\psi$ is a line meeting the exceptional locus of $\sigma'$ in one point, that we denote by $R_P$. Set $Q_P:=\sigma(R_P)$, denote by $\Sigma_P$ the $(k+1)$-linear space $\sigma(\sigma'^{-1}(\sigma'(R_P)))$ and by $\Omega_P$ the family of lines in $\Sigma_P$ passing by $Q_P$.
Since, by construction, $O\not\in\Sigma_P$, for every line $\ell\in\Omega_P$ the linear span $\langle O,\ell\rangle$ is a plane containing $\langle O,P\rangle$. Since the line $\ell$ gets contracted by $\psi$ it follows that $\ell$ meets $S$ in a subscheme of length at least two, and the plane $\langle O,\ell\rangle$ is trisecant. In particular, for every $P\in Z$ we may identify $\Omega_P$ with a $k$-dimensional linear subspace contained in $X$, whose union is a closed subset that we denote by $\Omega\subset X$.

Conversely, given a general trisecant plane $\Pi$ containing a line $\langle O,P\rangle$, with $S\cap\Pi=\{P,P',P''\}$, since $\psi_{|\Pi}$ is the standard Cremona trasformation with base points $\{P,P',P''\}$, then the point $Q_P$ coincides with $\langle O,P\rangle\cap \langle P',P''\rangle$ and $\langle P',P''\rangle\in\Omega_P$. This shows that $\Omega$ contains a dense open set of $X$, so that necessarily $\Omega=X$.

Finally, this may only occur if the obvious inclusion $\Omega_P\subseteq\pi'^{-1}(P)$ is an equality for all $P$. In fact, take an element $r$ in $\pi'^{-1}(P)$. It lies in $X=\Omega$, so it belongs to
 $\Omega_{P'}$ for some $P'\in Z$. The plane $\langle O,r\rangle$ contains $\langle O,P\rangle$, by hypothesis, and it contains
 a line $\ell$ contained in $\Sigma_{P'}$. But the strict transform of $\langle O,P\rangle$ in $B$
 meets the exceptional divisor of $\sigma'$ in one point $R_P$, whose image
 $Q_P$ via $\sigma$ is necessarily $\ell\cap OP$. This shows that
 $r$ belongs to $\Omega_P$.
\end{proof}

\bibliographystyle{amsalpha}

\end{document}